\documentclass[smallextended,envcountsame]{svjour3}
\smartqed
\usepackage{amsmath,amssymb}
\usepackage[dvipdfmx]{graphicx}
\usepackage{psfrag}
\usepackage{txfonts}
\usepackage{xcolor}
\usepackage[all]{xy}

\journalname{Japan J.\ Indust.\ Appl.\ Math.}

\newtheorem{assumption}[theorem]{Assumption}

\numberwithin{equation}{section}
\newcommand{\R}{\mathbb R}
\newcommand{\N}{\mathbb N}
\newcommand{\Ss}{\mathbf{S}}

\newcommand{\bfb}{\mathbf b}
\newcommand{\bfx}{\mathbf x}
\newcommand{\bfy}{\mathbf y}
\newcommand{\bfv}{\mathbf v}
\newcommand{\bfw}{\mathbf w}

\newcommand{\bfp}{\mathbf p}
\newcommand{\bfq}{\mathbf q}
\newcommand{\bfX}{\mathbf X}
\newcommand{\bfXi}{\mathbf\Xi}
\newcommand{\bnormal}{\mathbf n} 

\newcommand{\I}{\mathcal{I}}

\newcommand{\divv}{\mathrm{div}}

\newcommand{\T}{\mathcal{T}}
\newcommand{\E}{\mathcal{E}}
\newcommand{\hv}{\hat{v}}
\newcommand{\bu}{\bar{u}}

\newcommand{\Kab}{K_{\alpha\beta}}

\newcommand{\PP}{\mathcal{P}}
\newcommand{\dd}{\mathrm{d}}
\newcommand{\hK}{\widehat{K}}
\newcommand{\RT}{\mathcal{RT}}

\begin{document}

\title{Error Analysis of Crouzeix--Raviart and Raviart--Thomas Finite Element Methods}

\author{Kenta Kobayashi \and Takuya Tsuchiya}

\institute{Kenta Kobayashi \at
           Graduate School of Commerce and Management, \\
           Hitotsubashi University, Kunitachi, Japan \\
           \email{\texttt{kenta.k@r.hit-u.ac.jp}} \and
           Takuya Tsuchiya \at
           Graduate School of Science and Engineering, \\
           Ehime University, Matsuyama, Japan \\
           \email{\texttt{tsuchiya@math.sci.ehime-u.ac.jp}}
}

\date{Received: date / Accepted: date}

\maketitle

\begin{abstract}
We discuss the error analysis of the lowest degree Crouzeix--Raviart and
Raviart--Thomas finite element methods applied to a two-dimensional
Poisson equation.  To obtain error estimations, we use the
techniques developed by Babu\v{s}ka--Aziz and the authors. 
We present error estimates in terms of the circumradius and diameter
of triangles in which the constants are independent of the geometric
properties of the triangulations.
Numerical experiments confirm the results obtained.
\end{abstract}

\section{Introduction}
Let $\Omega \subset \R^2$ be a bounded polygonal domain, and
$\T_h$ be a triangulation of $\Omega$ consisting of triangular elements.
In this paper, we consider an error analysis of the Raviart--Thomas (RT)
and piecewise linear (nonconforming) Crouzeix--Raviart (CR) finite
element methods applied to the Poisson equation
\begin{align}
   - \Delta u = f \; \text{ in } \; \Omega, \qquad
    u = 0 \; \text{ on } \; \partial \Omega,
  \label{model-eq}
\end{align}
where $f \in L^2(\Omega)$ is a given function.
Let $u$ and $u_h^{CR}$ be the exact and CR finite
element solutions, respectively.
In standard text books, such as that by Brenner and Scott
\cite{BrennerScott}, the error of $u_h^{CR}$,
assuming $u \in H^2(\Omega)$,  is estimated as
\begin{align}
  \left\|u - u_h^{CR}\right\|_h := \left(\sum_{K \in \T_h}
   \int_K \left|\nabla u - \nabla u_h^{CR}\right|^2
   \dd \bfx\right)^{1/2}  \le C h |u|_{H^2(\Omega)},
  \label{standard-est}
\end{align}
where $h := \max_{K \in \T_h} h_K$, $h_K := \mathrm{diam}K$, and
$C$ is a constant independent of $u$ and $h$ but dependent on
the chunkiness parameter of the triangulations $\T_h$
\cite[Section~10.3]{BrennerScott}.  The dependence on the chunkiness
parameter in \eqref{standard-est} means that, if a triangulation
$\T_h$ contains very ``thin'' triangles, we cannot apply 
\eqref{standard-est}.
Note that the condition `$u \in H^2(\Omega)$' does not hold
in general, and we need to assume it explicitly.
See Assumption~\ref{reg-assump}.

A similar error estimation of the CR finite element method under
the \textit{maximum angle condition} was obtained in \cite{MaoShi},
in which the constant $C$ depends on the maximum angle of the
triangular elements.  Related error estimations were also discussed in
\cite{LiuKikuchi07}.

The aim of this paper is to show the estimation
\begin{equation}
  \left\|u - u_h^{CR}\right\|_h \le C (R + h) |u|_{H^2(\Omega)}
   \label{main-est}
\end{equation}
holds, where $R := \max_{K \in \T_h} R_K$, $R_K$ is the circumradius of
a triangle $K$, and the constant $C$ is independent of $u$ and $h$,
as well as the \textit{geometric properties of $\T_h$}.  Because
$C$ does not depend on the geometric properties of $\T_h$, we may
apply \eqref{main-est} even if $\T_h$ contains very ``skinny'' triangles.

Because the CR finite element method is non-conforming, a lemma similar
to C\'ea's lemma is not available, and this fact complicates the error
analysis of the CR finite element method.

To overcome this difficulty, we first consider the error analysis
of the RT finite element method. Because this method is conforming,
a C\'ea's-lemma-type claim is valid and we shall obtain error estimates
of its solutions (Theorem~\ref{RT-est}).
In the proof, we use techniques developed by Babu\v{s}ka and Aziz
\cite{BabuskaAziz} and the authors
\cite{KobayashiTsuchiya1,KobayashiTsuchiya3,KobayashiTsuchiya4,KobayashiTsuchiya5}.
It is well known that the CR and RT FEMs are related
\cite{ArnoldBrezzi,KikuchiSaito,LiuKikuchi07,Marini}, and an error
estimation of the CR FEM (Theorem~\ref{CR-error-est}) is obtained from that
of the RT FEM.

Finally, we present results of numerical experiments that are consistent
with the theoretical results obtained.

\section{Preliminaries}\label{Sect:Prelimi}
\subsection{Notation and function spaces}
Let $\R^2$ be the two-dimensional Euclidean space with Euclidean
norm $|\bfx| := (x_1^2 + x_2^2)^{1/2}$ for
$\bfx = (x_1,x_2)^\top \in \R^2$.
We always regard $\bfx \in \R^2$ as a column vector.
For a $2\times 2$ matrix $A$ and $\bfx \in \R^2$, $A^\top$ and
$\bfx^\top$ denote their transpositions. 
For a nonnegative integer $k$, let $\PP_k$ be the set of two-variable
polynomials with degrees of at most $k$.  

Let $\N_0$ be the set of nonnegative integers.
For $\delta = (\delta_1,\delta_2) \in \N_{0}^2$,
the multi-index $\partial^\delta$ of partial differentiation
(in the sense of the distribution) is defined by
\[
    \partial^\delta = \partial_\bfx^\delta
    := \frac{\partial^{|\delta|}\ }
   {\partial x_1^{\delta_1}\partial x_2^{\delta_2}}, \qquad
   |\delta| := \delta_1 + \delta_2.
\]
Sometimes $\partial^{(1,0)}v$ and $\partial^{(0,1)}v$ are denoted
by $v_x$ and $v_y$, respectively.  For a two-variable function $v$, its
gradient is denoted by $\nabla v = (v_x, v_y)$.  The gradient $\nabla v$
is regarded as a row vector.  Also, for a vector
$\bfw := (w_1,w_2)^\top$, its divergence $w_{1x} + w_{2y}$ is denoted by
$\nabla\cdot\bfw$ or $\divv\,\bfw$.  Note that $\nabla \bfw$ is a
$2\times 2$ matrix,
\begin{align*}
   \nabla \bfw = \begin{pmatrix} 
     w_{1x} & w_{1y} \\ w_{2x} & w_{2y}
   \end{pmatrix}.
\end{align*}

Let $\Omega \subset \R^2$ be a (bounded) domain.  The usual Lebesgue
space is denoted by $L^2(\Omega)$.
For a positive integer $k$, the Sobolev space $H^{k}(\Omega)$ is
defined by
$\displaystyle
  H^{k}(\Omega) := 
  \left\{v \in L^2(\Omega) \, | \, \partial^\delta v \in L^2(\Omega), \,
   |\delta| \le k\right\}$.
The norm and semi-norm of $H^{k}(\Omega)$ are defined by
\begin{gather*}
  |v|_{k,\Omega} := 
  \left(\sum_{|\delta|=k} \|\partial^\delta v\|_{L^2(\Omega)}^2
   \right)^{1/2}, \quad   \|v\|_{k,\Omega} := 
  \left(\sum_{0 \le m \le k} |v|_{m,\Omega}^2 \right)^{1/2}.
\end{gather*}
For a $2 \times 2$ matrix $A = (a_{ij})_{i,j=1,2}$, its Frobenius norm
$\|A\|_F$ is defined by $\|A\|_F^2 = \sum_{i,j=1,2} a_{ij}^2$.
Then, for $\bfw \in (H^1(\Omega))^2$, 
\begin{align*}
   |\bfw|_{1,\Omega}^2 = \int_\Omega \|\nabla \bfw\|_F^2\, \dd \bfx.
\end{align*}

The inner products of $L^2(\Omega)$ and $(L^2(\Omega))^2$ are denoted by
$(w,v)_\Omega$, $w$, $v \in L^2(\Omega)$, and
$(\bfw,\bfq)_\Omega$, $\bfw$, $\bfq \in (L^2(\Omega))^2$.
The space $H_0^1(\Omega)$ is
the closure of $C_0^\infty(\overline{\Omega}) \subset H^1(\Omega)$ with
respect to the topology of $H^1(\Omega)$.  We may use
$\|\nabla \phi\|_{0,\Omega} = |\phi|_{1,\Omega}$ as the norm of
$H_0^1(\Omega)$.  Then, $H^{-1}(\Omega)$ is the dual space of
$H_0^1(\Omega)$ with norm
\begin{align*}
   \|f\|_{-1,\Omega} := \sup_{\phi \in H_0^1(\Omega)}
    \frac{\langle f, \phi \rangle}{\|\nabla\phi\|_{0,\Omega}},
\end{align*}
where $\langle \cdot,\cdot \rangle$ is the duality pair of
$H^{-1}(\Omega)$ and $H_0^1(\Omega)$.
We also introduce the function space
\begin{align*}
   H(\divv,\Omega) := \left\{\bfw \in (L^2(\Omega))^2 \bigm|
    \divv\, \bfw \in L^2(\Omega) \right\}
\end{align*}
with norm
\begin{align*}
   \|\bfw\|_{H(\divv,\Omega)} := \left(\|\bfw\|_{0,\Omega}^2
     + \|\divv\, \bfw\|_{0,\Omega}^2\right)^{1/2}.
\end{align*}

\subsection{Model equation and its variational formulations}
\label{sec:2.2}
\label{sec:linear-trans}
The weak form of \eqref{model-eq} is
\begin{align}
   a(u,v) :=  (\nabla u, \nabla v)_\Omega 
   = (f,v)_\Omega =: \langle f,v\rangle, \quad
   \forall v \in H_0^1(\Omega).
  \label{weak-form}
\end{align}
By the Lax--Milgram lemma, there exists a unique solution
$u \in H_0^1(\Omega)$ for any
$f \in L^2(\Omega) \subset \; H^{-1}(\Omega)$.
From the definitions, the following inequality holds:
\begin{align}
   \|\nabla u\|_{0,\Omega} = \|f\|_{-1,\Omega}
    \le C_{1,1} \|f\|_{0,\Omega}, \qquad \forall f \in L^2(\Omega),
  \label{est-Hinv}
\end{align}
where the constant $C_{1,1}$ comes from Poincar\'e's inequality on $\Omega$.

We impose the following assumption on $\Omega$:
\begin{assumption}\label{reg-assump}
 For an arbitrary $f \in L^2(\Omega)$, the unique solution
$u$ of \eqref{model-eq} belongs to $H^2(\Omega)$, and
the following inequality holds,
\begin{align*}
   \|u\|_{2,\Omega} \le C_{1,2} \|f\|_{0,\Omega},
\end{align*}
where $C_{1,2}$ is a constant independent of $f$.
\end{assumption}
It is well known that if $\Omega$ is convex, then
Assumption~\ref{reg-assump} is valid \cite{Grisvard}.

The model equation \eqref{model-eq} has
a mixed variational formulation: \newline
Find $(\bfp,u) \in H(\divv,\Omega) \times L^2(\Omega)$ such that 
\begin{align}
\begin{split}
  (\bfp,\bfq)_\Omega + (u, \divv\, \bfq)_\Omega = 0, \qquad
  & \forall \bfq \in H(\divv,\Omega), \\
   (\divv\, \bfp, v)_\Omega + (f, v)_\Omega = 0, \qquad 
   & \forall v \in L^2(\Omega).
\end{split}
\label{mixed-eq}
\end{align}
The unique solvability of \eqref{mixed-eq} is equivalent to the
following inf-sup condition:
\begin{align}
   \inf_{v \in L^2(\Omega)}\sup_{\bfq \in H(\divv,\Omega)}
  \frac{(\divv\, \bfq,v)_\Omega}{\|v\|_{0,\Omega}\|\bfq\|_{H(\divv,\Omega)}}
   \ge \beta(\Omega) > 0.
   \label{infsup1}
\end{align}
It is easy to verify that 
$\beta(\Omega) := \left(1 + C_{1,1}^2\right)^{-1/2}$ satisfies the
inf-sup condition \eqref{infsup1}.  For the mixed variational
formulation for the model equation \eqref{model-eq}, readers are referred
to textbooks such as \cite{BoffiBrezziFortin}, \cite{Kikuchi}, and
\cite{KikuchiSaito}.

\subsection{Proper triangulation and the finite element methods}
Let $K \subset \R^2$ be a triangle with vertices $\bfx_i$, $i = 1,2,3$;
let $e_i$ be the edge of $K$ opposite to $\bfx_i$.  We always regard $K$
as a closed set.  A proper triangulation $\T_h$ of a bounded polygonal
domain $\Omega$ is a set of triangles that satisfies the conditions,
\begin{itemize}
 \item $\displaystyle \overline{\Omega} = \bigcup_{K \in \T_h} K$.
 \item If $K_1$, $K_2 \in \T_h$ with $K_1 \neq K_2$, we have either
   $K_1 \cap K_2 = \emptyset$ or $K_1 \cap K_2$ is a common vertex
   or a common edge.
\end{itemize}
With this definition, a proper triangulation $\T_h$ is sometimes
called a \textit{face-to-face} triangulation, and there exists no
hanging nodes in $\T_h$.  The fineness of $\T_h$ is indicated by
$h := \max_{K\in\T_h} h_K$, $h_K := \mathrm{diam}K$.  We denote the
set of edges in $\T_h$ by $\E_h$.  We also set
\begin{align*}
  \E_h^b := \{e \in \E_h \mid e \subset \partial\Omega \}, \qquad
  \E_h^i := \E_h \backslash \E_h^b.
\end{align*}

Let $e \in \E_h^i$ be shared by two triangles $K_1$ and $K_2$ with
$e = K_1 \cap K_2$.  Suppose that $v_h \in L^2(K_1 \cup K_2)$ satisfies
$v_h|_{K_i} \in \PP_1$.  Then, the \textit{jump} of $v_h$ on $e$ is
defined and denoted by
$[v_h] := \pm(\gamma_{K_1,e}(v) - \gamma_{K_2,e}(v)) $,
where $\gamma_{K_i,e}(v)$ $(i = 1, 2)$ is the trace operator for $v$ on
$K_i$ to the edge $e$, and the sign is taken arbitrarily and
fixed on each $e$.  Note that the sign of $[v_h]$ does not
affect the following definition of $S_h^{CR}$.
The finite element spaces for CR FEM are defined by
\begin{align*}
   S_h^{CR} & := \left\{v_h \in L^2(\Omega)
   \biggm| v_h|_K \in \PP_1, \forall K \in \T_h
  \text{ and } \int_e [v_h] \dd s = 0, \forall e \in \E_h^i\right\}, \\
   S_{h0}^{CR} & := \left\{v_h \in S_h^{CR} \biggm|
  \int_e v_h \dd s = 0, \forall e \in \E_h^b \right\}.
\end{align*}
Note also that, on $e \in \E_h^i$, functions in $S_h^{CR}$ are
continuous only at the midpoint of $e$.  The CR finite element solution
$u_h^{CR} \in S_{h0}^{CR}$ for the model equation is,
for $f \in L^2(\Omega)$,  defined by
\begin{align}
 a_h(u_h^{CR},v_h) := \sum_{K \in \T_h} \int_K \nabla u_h^{CR} \cdot
   \nabla v_h \dd \bfx = (f, v_h)_\Omega, \quad
   \forall v_h \in S_{h0}^{CR}.
   \label{CR-fem}
\end{align}
The norm associated with the bilinear form $a_h(\cdot,\cdot)$
is defined by $\|v_h\|_h := a_h(v_h,v_h)^{1/2}$ for $v_h \in S_h^{CR}$.

Regarding $\bfx \in \R^2$ as variables, let $\RT_0 \subset (\PP_1)^2$ be
defined by 
\begin{align*}
   \RT_0 := \{a \bfx + \mathbf{b} \,
    |\, \mathbf{b} \in \R^2, a \in  \R\} \subset (\PP_1)^2.
\end{align*}
For the RT finite element method, the finite element spaces $\Ss_h^{RT}$
and $S_h^C$ are defined by
\begin{align*}
   \Ss_h^{RT} & :=  \left\{ \bfp_h \in (L^2(\Omega))^2
    \Bigm| \bfp_h|_K \in \RT_0, \forall K \in \T_h \text{ and }
    \bfp_h \in H(\divv,\Omega) \right\}, \\
   S_h^{C} & := \{ v_h \in L^2(\Omega)  \bigm|
     v_h|_K \in \PP_0, \forall K \in \T_h \}.
\end{align*}
For a vector field $\bfq \in (H^{1}(K))^2$, its RT
interpolation $\I_K^{RT}\bfq$ on each $K \in \T_h$ is defined by 
\begin{align*}
   \int_{e_i} \left(\bfq - \I_K^{RT} \bfq\right)\cdot
   \bnormal \, \dd s = 0, \qquad i = 1,2,3,
\end{align*}
where $\bnormal$ is the unit outer normal vector on $\partial K$.
As $\mathrm{dim} \RT_0 = 3$, $\I_K^{RT} \bfq$ is determined uniquely.
Note also that $(\I_K^{RT}\bfq)\cdot\bnormal$ is a constant on
each $e_i$, $i=1,2,3$. Then, the global RT interpolation
$\I_h^{RT} \bfq \in \Ss_h^{RT}$ is defined as
$\I_h^{RT} \bfq\big|_K = \I_K^{RT} \bfq$ for each $K \in \T_h$.

Similarly, we define the projection $\pi_K^0$ on each $K$ by
\begin{align*}
\pi_K^0 v := \bar{v} :=  \frac{1}{K}\int_K v \dd s \quad \text{ or }
 \quad  \int_K \left(v - \pi_K^0 v\right) \dd s = 0
\end{align*}
for $v \in L^2(K)$.  This projection is extended as
$\pi_\Omega^0 : L^2(\Omega) \to S_h^C$ by
$\pi_\Omega^0 v \big|_K = \pi_K^0 v$ on each $K \in \T_h$ for
$v \in L^2(\Omega)$.  Note that 
$\pi_\Omega^0 : L^2(\Omega) \to S_h^C$ is an orthogonal projection.

The RT finite element method for the mixed variational equation
\eqref{mixed-eq} is defined by
\begin{align}
\begin{split}
  (\bfp_h,\bfq_h)_\Omega + (u_h^{RT}, \divv\, \bfq_h)_\Omega = 0, \qquad
  & \forall \bfq_h \in \Ss_h^{RT}, \\
   (\divv\, \bfp_h, v_h)_\Omega + (f, v_h)_\Omega = 0, \qquad 
   & \forall v_h \in S_h^C.
\end{split}
  \label{RT-fem}
\end{align}
Note that the RT FEM is conforming because
$\Ss_h^{RT}\times S_h^C \subset H(\divv,\Omega)\times L^2(\Omega)$.
Therefore, we may insert $(\bfq_h,v_h) \in \Ss_h^{RT} \times S_h^C$
into \eqref{mixed-eq} and take the difference between \eqref{mixed-eq}
and \eqref{RT-fem}, which implies
\begin{align}
\begin{split}
 (\bfp - \bfp_h,\bfq_h)_\Omega + (u - u_h^{RT}, \divv\, \bfq_h)_\Omega = 0,
 \qquad & \forall \bfq_h \in \Ss_h^{RT}, \\
   (\divv (\bfp - \bfp_h), v_h)_\Omega = 0, \qquad 
   & \forall v_h \in S_h^C.
\end{split}
\label{mixed-diff}
\end{align}
In regard to the convergence of the RT finite element solution, we must
consider the discrete inf-sup condition 
\begin{align*}
   \inf_{v_h \in S_h^C}\sup_{\bfq_h \in \Ss_h^{RT}}
  \frac{(\divv\, \bfq_h,v_h)_\Omega}
 {|v_h|_{0,\Omega}\|\bfq_h\|_{H(\divv,\Omega)}}
   \ge \beta_*,
\end{align*}
where $\beta_*$ is a constant independent of $h > 0$.
This point will be considered in Section~\ref{sec:inf-sup}.

\subsection{Relationship between the CR and RT
  finite element methods}
\label{sec:relationship}
It is well-known that the CR and RT finite element methods are closely
related.  Consider the following finite element equations,
\begin{align}
  \bu_h^{CR} \in S_h^{CR} \;  \text{ such that } \;
  a_h\left(\bu_h^{CR}, v_h\right) = \left(\pi_\Omega^0{f},v_h\right)_\Omega,
   \qquad \forall v_h \in S_{h0}^{CR}, 
   \label{CR-aux}\\
 (\bar\bfp_h,\bu_h^{RT}) \in \Ss_h^{RT}\times S_h^C \;  \text{ such that
 } \; 
\begin{split}
  (\bar{\bfp}_h,\bfq_h)_\Omega + (\bu_h^{RT}, \divv\, \bfq_h)_\Omega = 0, \quad
  & \forall \bfq_h \in \Ss_h^{RT}, \\
   (\divv\, \bar{\bfp}_h, v_h)_\Omega + (\pi_\Omega^0{f}, v_h)_\Omega = 0, \quad 
   & \forall v_h \in S_h^C.
\end{split}
  \label{RT-aux}
\end{align}
Then, on each $K \in \T_h$, the equalities
\begin{align}
  \bu_h^{RT} = \pi_K^0 \bu_h^{CR} + \frac{\pi_K^0 f}{48}
  \sum_{i=1}^3 |\bfx_i - \bfx_K|^2, \qquad
  \bar\bfp_h = \nabla \bu_h^{CR} - \frac{\pi_K^0 f}{2}(\bfx - \bfx_K),
  \label{CR-RT}
\end{align}
hold; here $\bfx_K := (\bfx_1 + \bfx_2 + \bfx_3)/3$ is the center of 
gravity of $K$. For details, readers are referred to
\cite{ArnoldBrezzi}, \cite{KikuchiSaito}, \cite{LiuKikuchi07},
\cite{Marini}.

\subsection{Linear transformations of triangles}
Let $\hK$ be the triangle with vertices $(0,0)^\top$, $(1,0)^\top$,
$(0,1)^\top$.  This $\hK$ is called the \textit{reference triangle}.
Let $\alpha \ge \beta > 0$ and $s^2 + t^2 = 1$, $t > 0$.
An arbitrary triangle on $\R^2$ is transformed to the triangle $K$ with
vertices $\bfx_1:=(0,0)^{\top}$, $\bfx_2:=(\alpha,0)^{\top}$,
$\bfx_3:=(\beta s,\beta t)^{\top}$ by a sequence of parallel
translations, rotations, and mirror imaging. 
Let $\Kab$ be the triangle with vertices $(0,0)^\top$, $(\alpha,0)^\top$,
$(0,\beta)^\top$.

We define the $2 \times 2$ matrices as 
\begin{align}
  A := \begin{pmatrix}
	   1 &  s \\ 0 &  t 
	 \end{pmatrix}, \quad
    B := A^{-1} = \begin{pmatrix}
	   1 & - st^{-1} \\ 0 & t^{-1}
	 \end{pmatrix}, \quad
    D_{\alpha\beta} :=  \begin{pmatrix}
	   \alpha & 0 \\ 0 & \beta
	 \end{pmatrix}.
  \label{matrixA}
\end{align}
Then, $\hK$ is transformed to $K$ and $\Kab$ by the transformations
$\bfy = AD_{\alpha\beta}\bfx$ and $\bfy = D_{\alpha\beta}\bfx$,
respectively.  Also, $\Kab$ is transformed to $K$ by $\bfy = A\bfx$.
Moreover, any function $w \in H^1(K)$ is pulled-back to a function
$v \in H^1(\Kab)$ as $H^1(K) \ni w \mapsto v := w\circ A \in H^1(\Kab)$.
A simple computation shows that $BB^\top$ has eigenvalues $(1\mp|s|)/t^2$.
Hence, the chain rule of differentiation implies that
$\nabla_\bfx v = (\nabla_\bfy w) B$,
$\,|\nabla_\bfx v|^2 = |(\nabla_\bfy w) B|^2$, and
\begin{align*}
 \frac{1-|s|}{t^2} |\nabla_\bfy w|^2 \le
   |\nabla_\bfx v|^2 \le \frac{1+|s|}{t^2} |\nabla_\bfy w|^2.
\end{align*}
With $\det A = t$, we have $|v|_{0,K}^2 = t |w|_{0,\Kab}^2$ and 
$\frac{1-|s|}{t} |v|_{1,\Kab}^2 \le |w|_{1,K}^2$.
Therefore, we obtain
\begin{gather}
   \frac{|w|_{0,K}^2}{|w|_{1,K}^2} \le
   \frac{t^2|v|_{0,\Kab}^2}
     {(1-|s|)|v|_{1,\Kab}^2}
  = (1 + |s|)\frac{|v|_{0,\Kab}^2}{|v|_{1,\Kab}^2}.
     \label{general-est}
\end{gather}

\subsection{Piola transformation}
To transform the vector fields, we need to introduce the Piola
transformation induced by an affine linear transformation
$\bfy = \varphi(\bfx) := A\bfx + \bfb$.  Suppose that a triangle
$\widetilde{K}$ is mapped to $K$ as $K := \varphi(\widetilde{K})$.
Then, the \textbf{Piola transformation} is the pull-back of the 
vector field $\bfp(\bfy)$ ($\bfy \in K$) to a vector
field $\bfq(\bfx)$,
\begin{align*}
   \bfq(\bfx) := A^{-1}\bfp(\varphi(\bfx)) = A^{-1}\bfp(\bfy).
\end{align*}
By the chain rule, we have
\begin{align*}
  \nabla_\bfx \bfq(\bfx) = A^{-1}(\nabla_\bfy \bfp(\varphi(\bfx))) A.
\end{align*}
By a straightforward computation, we confirm the following lemma is
valid.

\begin{lemma}\label{lem:Piola}
Let $A$ be a $2\times 2$ regular matrix and $\bfb \in \R^2$.
Suppose that a triangle $\widetilde{K} \subset \R^2$ is transformed
to $K$ by the affine linear transformation
$\varphi(\bfx) = A \bfx + \bfb:$ $K = \varphi(\widetilde{K})$.
Let $\tilde{e}_i$ be the edges of $\widetilde{K}$, and
$e_i := \varphi(\tilde{e}_i)$, $(i = 1, 2, 3)$.  Suppose that
a vector field $\bfp$ and a function $v$ on $K$ are pulled-back
to $\bfq(\bfx) := A^{-1}\bfp(\varphi(\bfx))$ and
$\tilde{v}(\bfx) := v(\varphi(\bfx))$, respectively.
Then, the following equalities holds:
\begin{align*}
 \int_K v \,\divv\, \bfp\, \dd \bfy & =
   (\det A) \int_{\widetilde{K}}
   \tilde{v}\, \divv\, \bfq\, \dd \bfx, \\
 \int_K \nabla_\bfy v \cdot \bfp\, \dd \bfy & =
   (\det A) \int_{\widetilde{K}}
    \nabla_\bfx \tilde{v} \cdot \bfq\, \dd \bfx, \\
 \int_{e_i} \bfp\cdot \bnormal\, \dd s & =
   (\det A) \int_{\tilde{e}_i}
    \bfq \cdot \tilde\bnormal\, \dd s, \quad i = 1,2,3,
\end{align*}
where $\bnormal$ and $\tilde\bnormal$ are the unit outer normal
vectors of $K$ and $\widetilde{K}$, respectively.
\end{lemma}

\section{Error analysis of the CR and RT finite element methods}
\subsection{Babu\v{s}ka--Aziz's technique}
In this section, we use the technique introduced by
Babu\v{s}ka and Aziz to claim that squeezing the reference triangle
perpendicularly does not 
reduce properties of the approximation of the interpolations.

Let $\Xi_2^\gamma \subset H^{1}(\hK)$ be defined by,
for $\gamma = (1,0)$ or $(0,1)$,
\begin{align*}
  \Xi_2^{(1,0)} & := \left\{ v \in H^{1}(\hK) \Bigm| 
    \int_0^1 v(s,0) \dd s
   = \int_{e_3} v \dd s = 0 \right\}, \\
  \Xi_2^{(0,1)} & := \left\{ v \in H^{1}(\hK) \Bigm| 
    \int_0^1 v(0,s) \dd s
   = \int_{e_2} v \dd s = 0 \right\}.
\end{align*}
Then, the constant $A_2$ is defined by
\begin{align*}
 A_2 := \sup_{v \in \Xi_2^{(1,0)}} \frac{|v|_{0,\hK}}{|v|_{1,\hK}}
  = \sup_{v \in \Xi_2^{(0,1)}} \frac{|v|_{0,\hK}}{|v|_{1,\hK}}, 
\end{align*}
and called the \textbf{Babu\v{s}ka--Aziz constant}.
According to Liu--Kikuchi \cite{LiuKikuchi},
$A_2$ is the maximum
positive solution of the equation $1/x + \tan(1/x) = 0$, and
$A_2 \approx 0.49291$. For the Babu\v{s}ka-Aziz constant,
the following lemma is known.
\begin{lemma}[Babu\v{s}ka--Aziz\cite{BabuskaAziz}]
\label{L4.1} $A_2 < \infty$.
\end{lemma}

Similarly, for $\hK$ and $\Kab$, we define the following sets:
\begin{align*}
  \bfXi_2 & := \left\{ \bfv \in (H^{1}(\hK))^2 \Bigm| 
    \int_{e_i} \bfv\cdot\bnormal \dd s = 0, \; i = 2,3 \right\}, \\
  \bfXi_2^{\alpha\beta} & := \left\{ \bfv \in (H^{1}(\Kab))^2 \Bigm| 
    \int_{e_i} \bfv\cdot\bnormal \dd s = 0, \; i = 2,3 \right\},
\end{align*}
Note that $\bfv = (v_1,v_2)^\top \in \bfXi_2$ if and only if
$v_1 \in \Xi_2^{(0,1)}$ and $v_2 \in \Xi_2^{(1,0)}$.  We thus
realize that
\begin{align}
  \sup_{\bfq \in \bfXi_2}
    \frac{|\bfq|_{0,\hK}}{|\bfq|_{1,\hK}} = A_2 < \infty.
  \label{lemma-const}
\end{align}
For $K \subset \R^2$, moreover, we define the following sets:
{\allowdisplaybreaks
\begin{align*}
  X_2^{(1)}(K) & := \left\{ v \in H^{1}(K) \Bigm| 
    \int_{K} v\, \dd\bfx  = 0 \right\}, \\
  \bfX_2^{(2)}(K) & := \left\{ \bfq \in (H^{1}(K))^2 \Bigm| 
    \int_{e_i} \bfq\cdot\bnormal\, \dd s = 0, \; i=1,2,3 \right\}.
\end{align*}
}
From the definitions, we obviously have 
$\bfX_2^{(2)}(\hK) \subset \bfXi_2$,
and $\bfX_2^{(2)}(\Kab) \subset \bfXi_2^{\alpha\beta}$.
Hence, by \cite[Lemma~4.2]{KobayashiTsuchiya5} and \eqref{lemma-const}, 
the following lemma holds.
\begin{lemma}\label{L4.2}
The constants
\begin{align*}
   B_2^{(1)}(\hK) & := \sup_{v \in X_2^{(1)}(\hK)}
   \frac{|v|_{0,\hK}}{|v|_{1,\hK}} < \infty, \\
  B_2^{(2)}(\hK) & := \sup_{\bfq \in \bfX_2^{(2)}(\hK)}
    \frac{|\bfq|_{0,\hK}}{|\bfq|_{1,\hK}} \le
    \sup_{\bfq \in \bfXi_2}
    \frac{|\bfq|_{0,\hK}}{|\bfq|_{1,\hK}} = A_2 < \infty,
\end{align*}
that indicate the approximation efficiency of
several interpolations on $\hK$ are bounded.
\end{lemma}

Let an arbitrary $v \in X_2^{(1)}(\Kab)$ be pulled-back to
$\hv := v\circ D_{\alpha\beta} \in X_2^{(1)}(\hK)$.  We immediately
note that
\begin{align*}
   |v|_{0,\Kab}^2 = \alpha\beta |\hv|_{2,\hK}^2, \quad
   |v_x|_{0,\Kab}^2 = \frac{\beta}{\alpha} |\hv_x|_{0,\hK}^2, \quad
   |v_y|_{0,\Kab}^2 = \frac{\alpha}{\beta}  |\hv_y|_{0,\hK}^2,
\end{align*}
which yields
\begin{align*}
   \frac{|v|_{0,\Kab}^2}{|v|_{1,\Kab}^2}
   = \frac{|\hv|_{0,\hK}^2}
   {\frac{1}{\alpha^{2}}|\hv_x|_{0,\hK}^2 
  + \frac{1}{\beta^{2}}|\hv_y|_{0,\hK}^2} 
 \le (\max\{\alpha,\beta\})^2\frac{|\hv|_{0,\hK}^2}{|\hv|_{1,\hK}^2}
   \le (\max\{\alpha,\beta\})^2 B_2^{(1)}(\hK)^2.
\end{align*}
Therefore, we obtain
\begin{align*}
 B_2^{(1)}(\Kab) = \sup_{v \in X_2^{(1)}(\Kab)} \frac{|v|_{0,\Kab}}{|v|_{1,\Kab}}
   \le \max\{\alpha,\beta\} B_2^{(1)}(\hK).
\end{align*}
Noting $\bfX_2^{(2)}(\Kab) \subset \bfXi_2^{\alpha\beta}$,
we similarly obtain 
\begin{align*}
 B_2^{(2)}(\Kab) := \! \! \sup_{\bfq \in \bfX_2^{(2)}(\Kab)}
 \frac{|\bfq|_{0,\Kab}}{|\bfq|_{1,\Kab}}
 \le \! \sup_{\bfq \in \bfXi_2^{\alpha\beta}}
 \frac{|\bfq|_{0,\Kab}}{|\bfq|_{1,\Kab}} 
  \le \max\{\alpha,\beta\}\sup_{\hat\bfq \in \bfXi_2}
 \frac{|\hat\bfq|_{0,\hK}}{|\hat\bfq|_{1,\hK}} 
  = \max\{\alpha,\beta\} A_2.
\end{align*}
Gathering the above inequalities, we obtain the following lemma.
\begin{lemma}\label{lem:squeezing}
The following inequalities hold:
\begin{align*}
   B_2^{(1)}(\Kab) \le \max\{\alpha,\beta\} B_2^{(1)}(\hK),
     \qquad
  B_2^{(2)}(\Kab) \le \max\{\alpha,\beta\} A_2.
\end{align*}
\end{lemma}

Lemma~\ref{lem:squeezing} means that squeezing the reference triangle
$\hK$ perpendicularly does not diminish the effectiveness of the
approximation through the interpolations on $\Kab$.

\subsection{Estimations on the general triangle.}
As stated in Section~\ref{sec:linear-trans}, an arbitrary triangle on
$\R^2$ is transformed to the triangle $K$ with vertices
$\bfx_1:=(0,0)^{\top}$, $\bfx_2:=(\alpha,0)^{\top}$,
$\bfx_3:=(\beta s,\beta t)^{\top}$ by a sequence of parallel
translations, rotations, and mirror imaging, where $s^2 + t^2 = 1$,
$0 < \beta \le \alpha$, and $t > 0$.  Then, $K$ is obtained from
$\Kab$ by the linear transformation $\bfy = A\bfx$, where $A$
is the matrix defined in \eqref{matrixA}.  Let $w \in H^1(K)$
be pulled-back to $v \in H^1(\Kab)$ as $v(\bfx) = w(A\bfx)$.
Combining \eqref{general-est} and Lemma~\ref{lem:squeezing}, we have
\begin{align*}
    B_2^{(1)}(K) := \sup_{v \in X_2^{(1)}(K)}
   \frac{|v|_{0,K}}{|v|_{1,K}} \le \sqrt{2}
   \sup_{u \in X_2^{(1)}(\Kab)}
   \frac{|u|_{0,\Kab}}{|u|_{1,\Kab}} 
    \le \sqrt{2} h_K B_2^{(1)}(\hK).
\end{align*}
Here, we used the assumption 
$0 < \beta \le \alpha \le h_K := \mathrm{diam}K$.

To consider $B_2^{(2)}(K)$, we need to introduce the Piola transformation
$\bfq(\bfx) := A^{-1}\bfp(A\bfx)$ for $\bfp \in \bfX_2^{(2)}(K)$.
By Lemma~\ref{lem:Piola}, we have $\bfq \in \bfX_2^{(2)}(\Kab)$ and
\begin{align*}
 |\bfp|_{0,K}^2 & = \int_K |\bfp(\bfy)|^2 \dd\bfy
 = (\det A)\int_{\Kab} |\bfp(A\bfx)|^2 \dd\bfx \\
 & = (\det A)\int_{\Kab} |A\bfq(\bfx)|^2 \dd\bfx
 \le (1 + |s|)t \int_{\Kab} |\bfq(\bfx)|^2 \dd\bfx 
 \le (1 + |s|)t |\bfq|_{0,\Kab}^2.
\end{align*}
In the above inequalities, we used the fact that the singular values
of $A$ are $(1 \pm |s|)^{1/2}$ and
$(1 - |s|)|\bfq|^2 \le |A\bfq|^2 \le (1 + |s|)|\bfq|^2$.

Let $Y$ and $Y$ be $2 \times 2$ matrices with $Y := AXA^{-1}$.
Given the singular values of $A$ and $A^{-1}$, we have
\begin{align*}
   \|Y\|_F^2 = \|AXA^{-1}\|_F^2
   \ge (1 - |s|)\|XA^{-1}\|_F^2
   \ge \frac{(1 - |s|)^2}{t^2} \|X\|_F^2.
 \end{align*}
Hence, it follows from 
$\nabla_\bfx\bfq(\bfx) = A^{-1}\nabla_\bfy\bfp(A\bfx)A$ that
\begin{align*}
  |\bfp|_{1,K}^2 & = \int_K \|\nabla_\bfy \bfp(\bfy)\|_F^2 \dd \bfy
  = (\det A)\int_{\Kab} \|\nabla_\bfy \bfp(A\bfx)\|_F^2 \dd \bfx \\
  & = t\int_{\Kab} \|A(\nabla_\bfx \bfq(\bfx))A^{-1}\|_F^2 \dd \bfx 
%
  \ge \frac{(1 - |s|)^2}{t} |\bfq|_{1,\Kab}^2.
\end{align*}
We thus obtain
\begin{align*}
 \frac{|\bfp|_{0,K}^2}{|\bfp|_{1,K}^2}
  \le \frac{(1+|s|) t^2}{(1 - |s|)^2}
   \frac{|\bfq|_{0,\Kab}^2}{|\bfq|_{1,\Kab}^2}
   = \frac{(1+|s|)^2}{1 - |s|}
   \frac{|\bfq|_{0,\Kab}^2}{|\bfq|_{1,\Kab}^2}.
\end{align*}

Assuming that the edge connecting $\bfx_2$ and $\bfx_3$ is the longest
edge of $K$, the following inequality holds
\cite[Lemma~3.2]{KobayashiTsuchiya3}:
\begin{align}
   \frac{1+|s|}{\sqrt{1-|s|}} \le 4\sqrt{2}\frac{R_K}{\alpha},
     \label{kobayashi-est}
\end{align}
To show this inequality, we first confirm that the following
inequality is valid:
\begin{equation*}
   \sqrt{1+|s|} \le \sqrt{2}\sqrt{1 + \gamma^2
   - 2 \gamma s}, \qquad
    0 < \forall \gamma \le 1, \; -1 < \forall s \le \frac{\gamma}{2}.
\end{equation*}
We then insert $\gamma := \beta/\alpha$ and use the laws of sines and
cosines.

Combining \eqref{kobayashi-est} with  Lemma~\ref{lem:squeezing},
we obtain
\begin{align*}
 B_2^{(2)}(K) := \sup_{\bfp \in \bfX_2^{(2)}(K)}
 \frac{|\bfp|_{0,K}}{|\bfp|_{1,K}}
 & \le  \frac{1+|s|}{\sqrt{1 - |s|}}
   \sup_{\bfq \in \bfX_2^{(2)}(\Kab)}
   \frac{|\bfq|_{0,\Kab}}{|\bfq|_{1,\Kab}} \\
  & \le 4\sqrt{2}\frac{R_K}{\alpha}B_2^{(2)}(\Kab)  
  \le 4\sqrt{2} R_K A_2.
\end{align*}

\begin{theorem} \label{Const-Thm}
Let $K$ be the triangle with vertices 
$\bfx_1:=(0,0)^{\top}$, $\bfx_2:=(\alpha,0)^{\top}$,
$\bfx_3:=(\beta s,\beta t)^{\top}$ such that $0 < \beta \le \alpha$,
$s^2 + t^2 = 1$, $t > 0$.  Suppose that the edge connecting
$\bfx_2$ and $\bfx_3$ is the longest edge of $K$.
Then, there exist positive constants $C_2^{(i)}$, $i=1,2$ that
are independent of $K$ such that the following estimates hold:
\begin{align*}
   B_2^{(1)}(K) := \sup_{v \in X_2^{(1)}(K)}
   \frac{|v|_{0,K}}{|v|_{1,K}} \le C_2^{(1)} h_K, \quad
    B_2^{(2)}(K) := \sup_{\bfq \in \bfX_2^{(2)}(K)}
   \frac{|\bfq|_{0,K}}{|\bfq|_{1,K}} \le C_2^{(2)} R_K,
\end{align*}
where $h_K := \mathrm{diam}K$ and $R_K$ is the circumradius of $K$.
\end{theorem}

\noindent
\textit{Remark}: Because $A_2 \approx 0.49291$, we have
$C_2^{(2)} \approx 2.7883$.

\vspace{2mm}
\noindent
\textit{Remark}: We present an example for the error estimates
of $B_2^{(2)}(K)$.  Let $K$ be the triangle with vertices $(1,0)^\top$,
$(-1,0)^\top$, and $(0,h)^\top$.  Let
\begin{align*}
   u := xy - \frac{1 + h^2}{2h}x, \quad
   \bfq := \nabla u = \begin{pmatrix}
     y - \frac{1 + h^2}{2h} \\ x \end{pmatrix}.
\end{align*}
Then, it is straightforward to verify
\begin{align*}
   \int_{e_i} \bfq \cdot \bnormal \dd s = 0, \quad i = 1,2,3,
\end{align*}
that is, $\I_K^{RT}\bfq = \mathbf{0}$, and
\begin{gather*}
   |\bfq - \I_K^{RT}\bfq|_{0,K}^2 = \frac{(3+h^2)(1+h^2)}{12h}, \quad
   |\bfq - \I_K^{RT}\bfq|_{1,K}^2 = 2h, \\
  \frac{|\bfq - \I_K^{RT}\bfq|_{0,K}}{|\bfq - \I_K^{RT}\bfq |_{1,K}}
 \ge \frac{1 + h^2}{2\sqrt{6}h} = \frac{R_K}{2\sqrt{6}}.
\end{gather*}
Because $\bfq - \I_K^{RT}\bfq \in \bfX_2^{(2)}$,
this inequality means that
the estimate $B_2^{(2)}(K) \le C_2^{(2)}R_K$
in Theorem~\ref{Const-Thm} cannot be further improved; that is,
the parameter $R_K$ is the \textit{best} possible parameter to measure
the convergence of solutions.

\subsection{Poincar\'e--Wirtinger's inequality on triangles}
We consider the error analysis of the projection
$\pi_K^0:L^2(K) \to \PP_0$ on $K$ and its extension to 
$\pi_\Omega^0:L^p(\Omega) \to S_h^C$ on $\Omega$.

\vspace{3mm}
From the definition, we have $\int_K(f(\bfx) - \bar f) \dd \bfx = 0$,
and therefore $f - \bar{f} \in X_2^{(1)}(K)$.  Hence, we obtain
the following theorem from Theorem~\ref{Const-Thm}.
See also \cite[Corollary~4.4]{KobayashiTsuchiya5}.
\begin{theorem}\label{PW-ineq1}
\begin{align*}
  \left|f - \pi_K^0 f\right|_{0,K} & \le C_2^{(1)} h_K |f|_{1,K}, \qquad
    \forall f \in H^{1}(K), \quad \forall K \in \T_h, \\
  \left|f - \pi_\Omega^0{f}\right|_{0,\Omega} & \le
   C_2^{(1)} h |f|_{1,\Omega}, \qquad \forall f \in H^{1}(\Omega),
\end{align*}
where $h := \max_{K\in \T_h} h_K$, and 
$C_2^{(1)}$ is the constant appearing in Theorem~\ref{Const-Thm}.
\end{theorem}

\begin{corollary} \label{PW-ineq3}
For an arbitrary proper triangulation $\T_h$ of a bounded polygonal
domain $\Omega$, the following estimation holds:
\begin{align*}
   \left\|f - \pi_\Omega^0{f}\right\|_{-1,\Omega}
   \le C_2^{(1)} h|f|_{0,\Omega}, \qquad \forall f \in L^2(\Omega),
\end{align*}
where $C_2^{(1)}$ is the constant appearing in Theorem~\ref{Const-Thm}.
\end{corollary}
\begin{proof}Because the projection $\pi_\Omega^0:L^2(K) \to S_h^C$
is orthogonal, we have, for an arbitrary $\phi \in H_0^1(\Omega)$,
\begin{align*}
  \left(f - \pi_\Omega^0 f,\phi \right)_{\Omega}
 &  = \left(f - \pi_\Omega^0 f,\phi - \pi_\Omega^0 \phi \right)_{\Omega}\\
 &  \le \left|f - \pi_\Omega^0 f \right|_{0,\Omega}
      \left|\phi - \pi_\Omega^0 \phi \right|_{0,\Omega}
   \le C_2^{(1)}h |f|_{0,\Omega}\|\nabla\phi\|_{0,\Omega}, \\
\intertext{and}
 \|f - \pi_\Omega^0 f\|_{-1,\Omega} & :=
  \sup_{\phi \in H_0^1(\Omega)} 
 \frac{\left(f - \pi_\Omega^0 f,\phi \right)_{\Omega}}
   {\|\nabla\phi\|_{0,\Omega}} 
  \le C_2^{(1)}h |f|_{0,\Omega}.
\end{align*}
In the above inequality, we used the fact 
$|f - \pi_\Omega^0 f|_{0,\Omega} \le |f|_{0,\Omega}$.  $\square$
\end{proof}

\vspace{1mm}
The important feature in the above Theorem~\ref{PW-ineq1} and
Corollary~\ref{PW-ineq3} is that constant $C_2^{(1)}$ does
not depend on the geometry of the triangulation $\T_h$ at all.

\subsection{Error analysis of the RT interpolation}
By definition, the RT interpolation $\I_K^{RT}$ satisfies
\begin{align*}
  \bfq - \I_K^{RT}\bfq \in \bfX_2^{(2)}(K), \quad
  \bfq \in (H^1(\Omega))^2
\end{align*}
on each triangle $K \in \T_h$.  Therefore, it follows from
Theorem~\ref{Const-Thm} that, for each $K \in \T_h$,
\begin{align}
  |\bfq - \I_K^{RT}\bfq|_{0,K} \le C_2^{(2)} R_K 
  |\bfq - \I_K^{RT}\bfq|_{1,K}.
 \label{RT-1}
\end{align}
Moreover, the definition of $\bfX_2^{(2)}(K)$ and the divergence
theorem yield
\begin{align}
 \int_K \divv(\bfq - \I_K^{RT} \bfq) \dd \bfx =
  \int_{\partial K} (\bfq - \I_K^{RT} \bfq)\cdot \bnormal \dd s = 0.
     \label{michiko}
\end{align}
Note that $\divv(\I_K^{RT} \bfq) \in \PP_0$ because
$\I_K^{RT} \bfq \in \RT_0$.  Hence, we realize that
\begin{align}
  \divv (\I_K^{RT} \bfq) = \frac{1}{|K|} \int_K \divv\, \bfq \dd \bfx 
  = \pi_K^0(\divv\,\bfq).
    \label{RT-ineq1}
\end{align}
Setting constant $a_\bfq$ to $a_\bfq := \pi_K^0(\divv\,\bfq)/2$,
we then have
\begin{align*}
  \left(\I_K^{RT}\bfq\right)(\bfx) = a_\bfq \bfx + \bfb_\bfq,
  \quad \bfb_\bfq \in \R^2.
\end{align*}
Therefore, for $\bfq(\bfx) = (q_1(\bfx), q_2(\bfx))^\top$ and
$\bfx = (x, y)^\top$, we have
\begin{align*}
 |\bfq - \I_K^{RT}\bfq|_{1,K}^2
 & = |\bfq|_{1,K}^2 + 2 \int_K a_\bfq^2 \dd \bfx 
  - 2 a_\bfq \int_K \divv \bfq \dd \bfx \\
 & = |\bfq|_{1,K}^2 + 2 \int_K a_\bfq^2 \dd \bfx
  - 2 a_\bfq \int_K \divv \left(\I_K^{RT}\bfq\right) \dd \bfx \\
 & = |\bfq|_{1,K}^2 - 2 \int_K a_\bfq^2 \dd \bfx 
   \le |\bfq|_{1,K}^2.
\end{align*}
In the above inequalities, we used the equality \eqref{michiko}.
Gathering \eqref{RT-1}, \eqref{RT-ineq1}, Theorem~\ref{PW-ineq1},
and the above inequality, we obtain the following theorem.
\begin{theorem} \label{Thm-RT-int}
For an arbitrary $\bfq \in (H^1(\Omega))^2$
with $\divv \bfq \in H^1(\Omega)$,
the following estimates hold:
\begin{align*}
    \left|\bfq - \I_K^{RT}\bfq\right|_{0,K} & \le
  C_2^{(2)} R_K |\bfq|_{1,K}, & 
  \left|\divv\, \bfq - \divv\, \I_K^{RT} \bfq\right|_{0,K} 
   & \le C_2^{(1)} h_K |\divv\, \bfq|_{1,K}, \\
  \left|\bfq - \I_h^{RT}\bfq\right|_{0,\Omega} & \le
  C_2^{(2)} R |\bfq|_{1,\Omega},,&
   \left|\divv\, \bfq - \divv\, \I_h^{RT} \bfq\right|_{0,\Omega} 
  & \le C_2^{(1)} h |\divv\, \bfq|_{1,\Omega},
\end{align*}
where $R := \max_{K \in \T_h}R_K$, $h := \max_{K \in \T_h}h_K$,
and $C_2^{(i)}$, $i = 1, 2$ are the constants appearing in
Theorem~\ref{Const-Thm} that are independent of $\T_h$ and $\bfq$.
\end{theorem}

\subsection{Discrete inf-sup condition for the RT finite  elements}
\label{sec:inf-sup}
Following Mao--Shi \cite{MaoShi}, we now discuss the discrete inf-sup
condition for the RT finite elements.  Take $v_h \in S_h^C$
as arbitrary, and consider the following Poisson problem:
\begin{align*}
   - \Delta w = v_h \text{ in } \Omega,  \qquad
    w = 0 \text{ on }  \partial \Omega.
\end{align*}

Let $\bfq := - \nabla w$. 
Suppose that
Assumption~\ref{reg-assump} holds and $C_2^{(2)}R \le 1$.
Then, recalling that $\divv(\I_h^{RT}\bfq) = \pi_\Omega^0(\divv\,\bfq)$,
it follows from Theorems~\ref{PW-ineq1} and \ref{Thm-RT-int} that
\begin{align*}
  \|\I_h^{RT}\bfq\|_{H(\divv,\Omega)} & = 
  \left(|\I_h^{RT}\bfq|_{0,\Omega}^2 +
     |\divv(\I_h^{RT}\bfq)|_{0,\Omega}^2\right)^{1/2} \\
  & \le \left(2|\bfq|_{0,\Omega}^2
   + 2|\bfq - \I_h^{RT}\bfq|_{0,\Omega}^2 +
     |\divv(\I_h^{RT}\bfq)|_{0,\Omega}^2\right)^{1/2} \\
  & \le  \left(2|\bfq|_{0,\Omega}^2 + 2C_2^{(2)}{}^2R^2
   |\bfq|_{1,\Omega}^2
    + |\divv\,\bfq|_{0,\Omega}^2\right)^{1/2} \\
 & \le  \left(2\|w\|_{2,\Omega}^2
   + |v_h|_{0,\Omega}^2\right)^{1/2}
  \le \left(2 + C_{1,2}^2\right)^{1/2} |v_h|_{0,\Omega}.
\end{align*}
Moreover, we realize that
\begin{gather*}
  \left(\divv(\I_h^{RT}\bfq), v_h\right)_{\Omega}
  = \left(\pi_\Omega^0(\divv\,\bfq), v_h\right)_{\Omega}
  = (\divv\,\bfq, v_h)_{\Omega}
 = (v_h,v_h)_\Omega = |v_h|_{0,\Omega}^2,
\intertext{and therefore,}
 \frac{\left(\divv(\I_h^{RT}\bfq), v_h\right)_{\Omega}}
  {\|\I_h^{RT}\bfq\|_{H(\divv,\Omega)}}
  \ge \frac{|v_h|_{0,\Omega}}{\left(2 + C_{1,2}^2\right)^{1/2}}.
\end{gather*}
Hence, we finally conclude that
\begin{gather*}
     \sup_{\bfq_h \in \Ss_h^{RT}}
  \frac{(\divv\, \bfq_h,v_h)_\Omega}
 {\|\bfq_h\|_{H(\divv,\Omega)}}
   \ge \frac{|v_h|_{0,\Omega}}{\left(2 + C_{1,2}^2\right)^{1/2}},
\end{gather*}
and obtain the following theorem.

\begin{theorem} \label{RT-infsup}
On the regularity of the solutions of the model problem, we impose
Assumption~\ref{reg-assump}.  Then, for a triangulation $\T_h$ of
$\Omega$ that satisfies $C_2^{(2)}R \le 1$ with
$R := \max_{K \in \T_h} R_K$, the following discrete inf-sup condition
holds:
\begin{align}
    \inf_{v_h \in S_h^C}\sup_{\bfq_h \in \Ss_h^{RT}}
  \frac{(\divv\, \bfq_h,v_h)_\Omega}
 {|v_h|_{0,\Omega}\|\bfq_h\|_{H(\divv,\Omega)}}
   \ge \frac{1}{\left(2 + C_{1,2}^2\right)^{1/2}} =: \beta_*.
   \label{d-inf-sup}
\end{align}
Here, $C_2^{(2)}$ is the constant appearing in Theorem~\ref{Const-Thm},
and $C_{1,2}$ is the constant appearing in Assumption~\ref{reg-assump}.
\end{theorem}

\subsection{Error analysis of the RT finite element method}
Because of the inclusion 
$\Ss_h^{RT}\times S_h^C \subset H(\divv,\Omega)\times L^2(\Omega)$,
the RT finite element method \eqref{RT-fem} for the mixed
variational formulation \eqref{mixed-eq} is conforming, and the
following C\'ea's-lemma-type estimation is known
\cite[Lemma~2.44]{ErnGuermond}.

\begin{theorem}\label{RT-est}
Let $(\bfp,u) \in H(\divv,\Omega) \times L^2(\Omega)$ be the exact 
solution of \eqref{mixed-eq}, and 
$(\bfp_h,u_h^{RT}) \in \Ss_h^{RT} \times S_h^C$ be the solution
of the RT finite element method \eqref{RT-fem}. 
Suppose that Assumption~\ref{reg-assump} holds and
$C_2^{(2)}R \le 1$.  Then, we have the following
error estimations:
\begin{align*}
  |\bfp - \bfp_h|_{0,\Omega} & \le 2 (1 + \beta_{*}^{-1})
  \inf_{\bfq_h \in \Ss_h^{RT}}\|\bfp - \bfq_h\|_{H(\divv,\Omega)}, \\
  \left|u - u_h^{RT}\right|_{0,\Omega} & \le
  (1 + \beta_*^{-1}) \inf_{w_h \in S_h^C}|u - w_h|_{0,\Omega}
  + \beta_*^{-1}\inf_{\bfq_h \in \Ss_h^{RT}} |\bfp - \bfq_h|_{0,\Omega},
\end{align*}
where $\beta_*$ is the constant of the discrete inf-sup condition
appearing in \eqref{d-inf-sup}.
\end{theorem}

Gathering Theorems~\ref{PW-ineq1}, \ref{Thm-RT-int},
\ref{RT-infsup}, and \ref{RT-est}, we immediately obtain
the following corollary.

\begin{corollary} \label{RT-cor}
Under the assumptions of Theorem~\ref{RT-est},
we have the following error estimations:
\begin{align*}
  |\bfp - \bfp_h|_{0,\Omega}  \le C R |f|_{0,\Omega}, \qquad
  \left|u - u_h^{RT}\right|_{0,\Omega}  \le C (R+h) |f|_{0,\Omega},
\end{align*}
where constant $C$ depends on $C_2^{(i)}$, $i=1,2$ and $C_{1,2}$,
but is independent of $h$, $R$, $f$, and the geometric properties of $\T_h$.
\end{corollary}

\subsection{Error analysis of the CR finite element method}
In this section, we estimate the error $\|u - u_h^{CR}\|_h$, where
$u \in H^2(\Omega)\cap H_0^1(\Omega)$ is the exact solution 
of \eqref{weak-form}, and $u_h^{CR} \in S_{h0}^{CR}$ is the
CR finite element solution defined by \eqref{CR-fem}.
We impose Assumption~\ref{reg-assump}.

We introduce the following auxiliary equations: for $f \in L^2(\Omega)$,
\begin{align*}
    - \Delta \bu = \pi_\Omega^0 f \quad \text{ in } \Omega, \qquad
    \bu = 0 \quad \text{ on }  \partial \Omega.
\end{align*}
The CR FEM for this equation is defined by \eqref{CR-aux}.
Note that $u - \bu$ satisfies
\begin{align*}
  a(u - \bu, v) = (f - \pi_\Omega^0f, v)_\Omega, \qquad
   \forall v \in H_0^1(\Omega).
\end{align*}
Therefore, from \eqref{est-Hinv} and the Poincar\'e--Wirtinger inequality
(Corollary~\ref{PW-ineq3}), we have
\begin{align}
  |u - \bu|_{1,\Omega} \le C_{1,1}
  \left\|f - \pi_\Omega^0 f\right\|_{-1,\Omega} \le C_{1,1}
  C_2^{(1)} h |f|_{0,\Omega}.
  \label{CR-est1}
\end{align}

Because $\pi_\Omega^0:L^2(\Omega) \to S_h^C$ is an orthogonal projection,
we have, for an arbitrary $v_h \in S_h^{CR}$,
\begin{align*}
  \left(f - \pi_\Omega^0 f, v_h\right)_\Omega =
  \left(f - \pi_\Omega^0 f, v_h - \pi_\Omega^0 v_h\right)_\Omega, \qquad
  \left|f - \pi_\Omega^0 f\right|_{0,\Omega} \le |f|_{0,\Omega}.
\end{align*}
Hence, Poincar\'e--Wirtinger's inequality (Theoerm~\ref{PW-ineq1}) yields
\begin{align*}
 \left|v_h - \pi_\Omega^0 v_h\right|_{0,\Omega}^2
 & = \sum_{K \in \T_h} \int_K \left|v_h - \pi_K^0 v_h\right|^2 \dd  \bfx \\
 & \le C_2^{(1)}{}^2\sum_{K \in \T_h} h_K^2 \int_K \left|\nabla v_h\right|^2 \dd \bfx
   \le C_2^{(1)}{}^2 h^2 \|v_h\|_h^2
\end{align*}
and 
$\left|v_h - \pi_\Omega^0 v_h\right|_{0,\Omega} \le C_2^{(1)} h \|v_h\|_h$.
Inserting $v_h := u_h^{CR} - \bu_h^{CR} \in S_{h0}^{CR}$ into 
\begin{align*}
  a_h\left(u_h^{CR} - \bu_h^{CR}, v_h\right)
  = \left(f - \pi_\Omega^0 f, v_h\right)_\Omega
  = \left(f - \pi_\Omega^0 f, v_h - \pi_\Omega^0 v_h\right)_\Omega,
\end{align*}
we obtain
\begin{align*}
 \left\|u_h^{CR} - \bu_h^{CR}\right\|_h^2 & \le
 \left|f - \pi_\Omega^0 f\right|_{0,\Omega}
 \left|u_h^{CR} - \bu_h^{CR} - \pi_\Omega^0(u_h^{CR} -
 \bu_h^{CR})\right|_{0,\Omega} \\
 & \le C_2^{(1)}h \left|f - \pi_\Omega^0 f\right|_{0,\Omega}
   \|u_h^{CR} - \bu_h^{CR}\|_h,
\end{align*}
and
\begin{align}
 \left\|u_h^{CR} - \bu_h^{CR}\right\|_h \le 
  C_2^{(1)}h \left|f - \pi_\Omega^0 f\right|_{0,\Omega}
  \le C_2^{(1)}h |f|_{0,\Omega}.
  \label{CR-est2}
\end{align}

As explained in Section~\ref{sec:relationship},
$\bu_h^{CR}$ is written as
\begin{align*}
 \nabla \bu_h^{CR} =  \bar\bfp_h + \frac{\pi_K^0 f}{2}(\bfx - \bfx_K)
\end{align*}
on each $K \in \T_h$, where $\bar\bfp_h \in \Ss_h^{RT}$ is the
RT finite element solution defined by \eqref{RT-aux}, and
$\bfx_K$ is the center of gravity of $K$.  Setting
$\bar\bfp := \nabla \bu$, we then find
{\allowdisplaybreaks
\begin{align*}
 \left\|\bu - \bu_h^{CR}\right\|_h^2 & =
  \sum_{K \in \T_h} \int_K \left|\bar\bfp - \bar\bfp_h +
 (\pi_K^0f) (\bfx - \bfx_K)/2 \right|^2 \dd \bfx \\
 & \le 2 |\bar\bfp - \bar\bfp_h|_{0,\Omega}^2
  + \frac{1}{2}\sum_{K \in \T_h} (\pi_K^0f)^2 \int_K
    |\bfx - \bfx_K|^2 \dd \bfx \\
 & = 2 |\bar\bfp - \bar\bfp_h|_{0,\Omega}^2
  + \frac{1}{2}\sum_{K \in \T_h} (\pi_K^0f)^2 \frac{|K|}{12}
    \sum_{i=1}^3 |\bfx_i - \bfx_K|^2 \\
 & \le 2 |\bar\bfp - \bar\bfp_h|_{0,\Omega}^2
  + \frac{h^2}{24}\sum_{K \in \T_h} (\pi_K^0f)^2|K| \\
 & \le 2 |\bar\bfp - \bar\bfp_h|_{0,\Omega}^2
  + \frac{h^2}{24} \int_\Omega |f|^2 \dd \bfx.
\end{align*}
}
In the above inequalities, we used the result
\begin{align*}
  \int_K |\bfx - \bfx_K|^2 \dd \bfx
   = \frac{|K|}{12}\sum_{i=1} |\bfx_i - \bfx_K|^2 \le
    \frac{|K|}{12} h_K^2.
\end{align*}

Setting $\bar\bfp := \nabla \bu$, it follows from
Corollary~\ref{RT-cor} that 
\begin{align*}
 |\bar\bfp - \bar\bfp_h|_{0,\Omega} & \le C R
   \left|\pi_\Omega^0 f\right|_{0,\Omega} \le
  C R \left|f\right|_{0,\Omega}.
\end{align*}
Therefore, we finally conclude 
\begin{align}
  \left\|\bu - \bu_h^{CR}\right\|_h \le C(R + h) |f|_{0,\Omega}.
  \label{CR-est3}
\end{align}
Gathering the estimations \eqref{CR-est1}, \eqref{CR-est2}, \eqref{CR-est3}
with the triangle inequality, we obtain the following theorem.

\begin{theorem} \label{CR-error-est}
Let $u \in H^2(\Omega)\cap H_0^1(\Omega)$ be the exact solution of
\eqref{model-eq}, and $u_h^{CR} \in S_{h0}^{CR}$ be the
CR finite element solution \eqref{CR-fem}.
Suppose that Assumption~\ref{reg-assump} holds and
$C_2^{(2)}R \le 1$. Then, we have the following error estimations:
\begin{align*}
 \|u - u_h^{CR}\|_h \le C(R + h) |f|_{0,\Omega},
\end{align*}
where constant $C$ depends on $C_2^{(i)}$, $i=1,2$, 
$C_{1,1}$, and $C_{1,2}$, but is independent of $h$, $R$, $f$, and
the geometric properties of $\T_h$.
\end{theorem}

\section{Numerical experiments}
In this section, we present the results of numerical experiments that
confirm the obtained error estimations.
Let $\Omega := (0,1) \times (0,1)$.
We compute the $P_1$ Lagrange and the CR finite element solutions,
$u_h^L$ and $u_h^{CR}$, respectively, for the model problem
\begin{align*}
 -\Delta u &= 2x(1-x) + 2y(1-y) \quad \text{ in } \quad \Omega, \qquad
  u = 0 \quad  \text{ on } \quad \partial\Omega,
\end{align*}
which has the exact solution $u = x(1-x)y(1-y)$.

To this end, we triangulate $\Omega$ with triangles of height $1/N$ and
baseline length $1/M$ (see Figure~1), with a positive
integer $M$ and a positive and even integer $N$.
The triangulation has $(2M+1)N$ elements, and
the numbers of freedom are $MN + M + 3N/2 + 1$ and
$3MN + M + 5N/2$ for the $P_1$ Lagrange and the CR elements,
respectively.  In Figure~2, we give the finite element solutions obtained.
\begin{figure}[thbp]
 \centering
 \includegraphics[width=11cm]{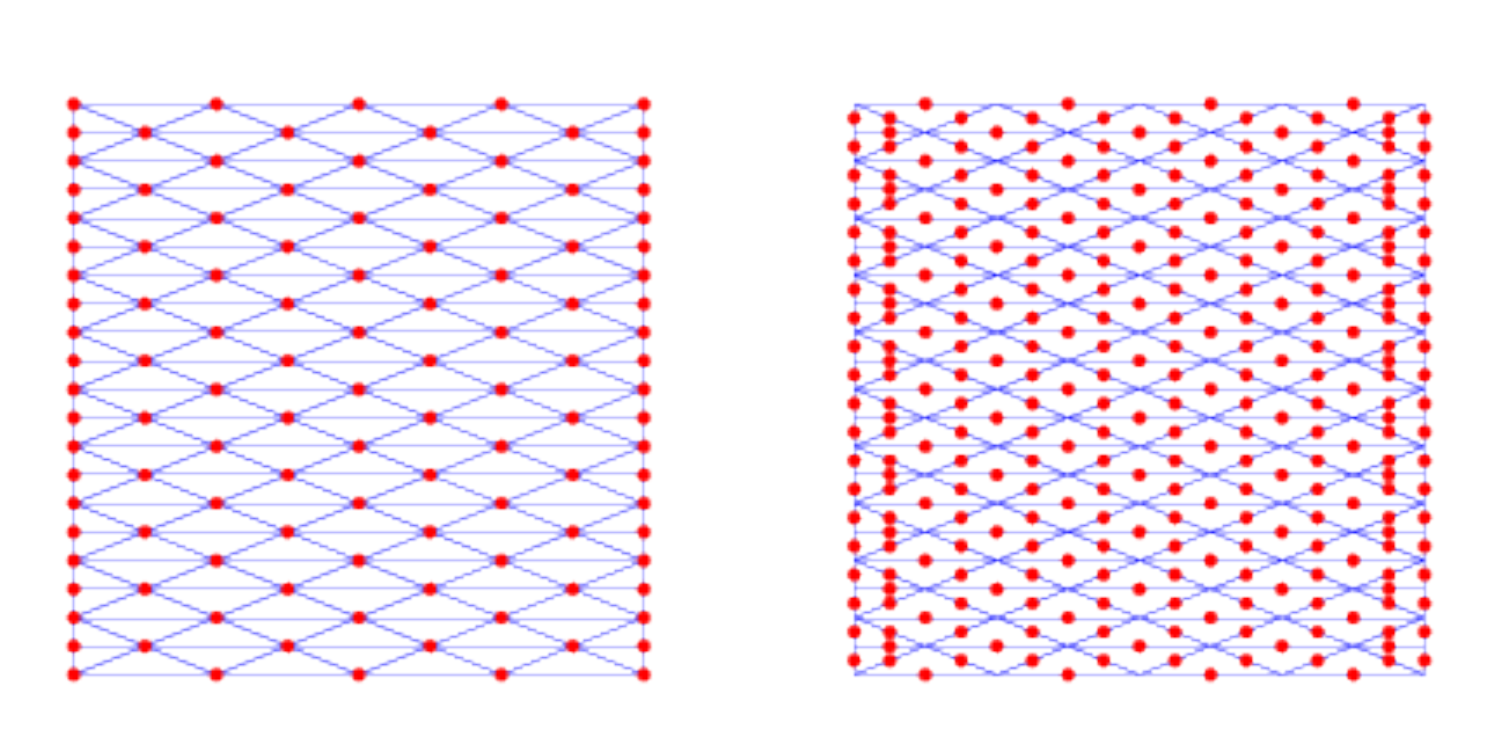}
\caption{Triangulation used for $\Omega$ with $M = 4$ and $N = 20$.
The dots on the left indicate the degrees of freedom of the $P_1$
Lagrange elements, and those on the right indicate the degrees of
freedom of the CR elements.}
\label{fig:mesh}
\end{figure}

\begin{figure}[thbp]
 \centering
 \includegraphics[width=10.3cm]{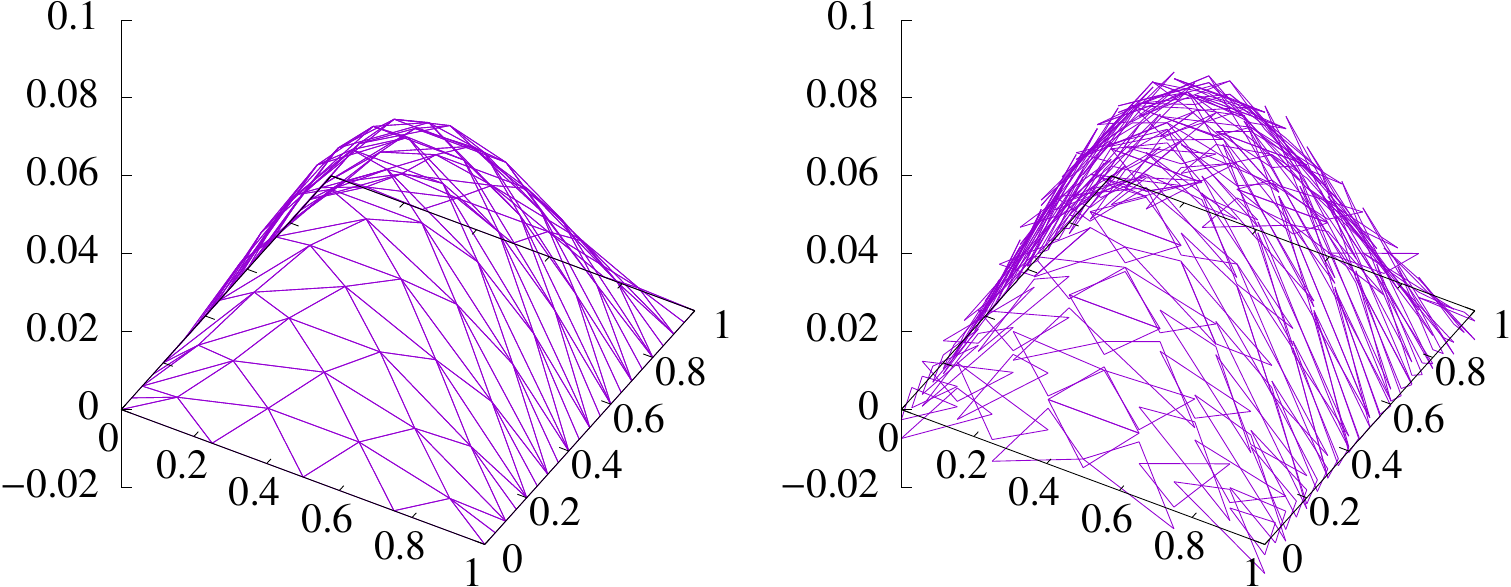}
\caption{Numerical solutions: $P_1$ Lagrange elements (left),
 CR elements (right).  Note that the
values of the finite element solutions are somewhat exaggerated.}
\label{fig:result}
\end{figure}

Setting $N$ to be the closest even integer to $M^\alpha$ with
$\alpha = 1.5$, we compute the error of $|u - u_h^L|_{1,\Omega}$
and $\|u - u_h^{CR}\|_h$ for various $M$.
The results are given in Tables~\ref{p1radius} and \ref{crradius}.
Note that because $C_2^{(2)} \approx 2.7883$ as stated before,
the condition $C_2^{(2)}R \le 1$ is satisfied if
$R \lessapprox 0.3586$.
We clearly see that the behavior of the errors is consistent
with \cite[Theorem~3]{KobayashiTsuchiya1} and 
\cite[Theorem~1.1]{KobayashiTsuchiya3} for $P_1$ Lagrange elements,
and Theorem~\ref{CR-error-est} for the CR finite elements.
In particular, we emphasize that 
$|u - u_h^{L}|_{1,\Omega}/R$ and $\|u - u_h^{CR}\|_h/R$ look stable,
whereas $|u - u_h^{L}|_{1,\Omega}/h$ and $\|u - u_h^{CR}\|_h/h$
seem to diverge.  This means that the error estimations in terms
of the circumradius $R$ are \textit{essential} and
\textit{the best possible}.

\begin{table}[htbp]
\centering
 \caption{Error $|u - u_h^L|_{1,2}$ with various $M$
 and $N \approx M^{1.5}$.}
\begin{tabular}{ccccccc}
\hline
 $M$ & $N$ & $h$ & $R$ & $|u - u_h^L|_{1,\Omega}$
 & $|u - u_h^L|_{1,\Omega}/h$
 & $|u - u_h^L|_{1,\Omega}/R$ \\
\hline
10  & 32   & 0.1000 & 0.1756 & 0.0167277 & 0.1672776 & 0.0952470 \\
20  & 90   & 0.0500 & 0.1180 & 0.0108223 & 0.2164462 & 0.0916713 \\
30  & 164  & 0.0333 & 0.0941 & 0.0085646 & 0.2569403 & 0.0909588 \\
40  & 252  & 0.0250 & 0.0807 & 0.0073229 & 0.2929178 & 0.0907044 \\
50  & 354  & 0.0200 & 0.0722 & 0.0065410 & 0.3270520 & 0.0905805 \\
60  & 464  & 0.0166 & 0.0655 & 0.0059329 & 0.3559770 & 0.0905489 \\
70  & 586  & 0.0142 & 0.0606 & 0.0054905 & 0.3843359 & 0.0905290 \\
80  & 716  & 0.0125 & 0.0566 & 0.0051271 & 0.4101745 & 0.0905289 \\
90  & 854  & 0.0111 & 0.0533 & 0.0048257 & 0.4343167 & 0.0905366 \\
100 & 1000 & 0.0100 & 0.0505 & 0.0045726 & 0.4572635 & 0.0905472 \\
\hline
\end{tabular}
\label{p1radius}
\end{table}

\begin{table}[htbp]
\centering
 \caption{Error $\|u - u_h^{CR}\|_{h}$ with various $M$
 and $N \approx M^{1.5}$.}
\begin{tabular}{ccccccc}
\hline
 $M$ & $N$ & $h$ & $R$ & $\|u - u_h^{CR}\|_h$
&  $\|u - u_h^{CR}\|_h/h$ & $\|u - u_h^{CR}\|_h/R$ \\
\hline
10  & 32   & 0.1000 & 0.1756 & 0.0167791 & 0.1677918 & 0.0955398 \\
20  & 90   & 0.0500 & 0.1180 & 0.0104671 & 0.2093425 & 0.0886627 \\
30  & 164  & 0.0333 & 0.0941 & 0.0081263 & 0.2440346 & 0.0863037 \\
40  & 252  & 0.0250 & 0.0807 & 0.0068669 & 0.2746769 & 0.0850560 \\
50  & 354  & 0.0200 & 0.0722 & 0.0060827 & 0.3041381 & 0.0842343 \\
60  & 464  & 0.0166 & 0.0655 & 0.0054908 & 0.3294498 & 0.0838012 \\
70  & 586  & 0.0142 & 0.0606 & 0.0050614 & 0.3543016 & 0.0834546 \\
80  & 716  & 0.0125 & 0.0566 & 0.0047136 & 0.3770886 & 0.0832266 \\
90  & 854  & 0.0111 & 0.0533 & 0.0044273 & 0.3984651 & 0.0830631 \\
100 & 1000 & 0.0100 & 0.0505 & 0.0041883 & 0.4188380 & 0.0829382 \\
\hline
\end{tabular}
\label{crradius}
\end{table}

\vspace{0.3cm}
\textsc{Acknowledgments.}
The authors were supported by JSPS KAKENHI Grant Numbers
JP26400201, JP16H03950, and JP17K18738.
The authors thank the anonymous referee for the
valuable comments that helped to improve this paper.

\end{document}